\documentclass[12pt]{amsart}
\usepackage{ifthen,verbatim}
\usepackage{floatflt,graphicx,graphics}

 \usepackage{tikz} 
\usepackage{mathrsfs}
\usepackage{amsmath,amssymb,amsthm}
\usepackage{mathtools}

\numberwithin{equation}{section}
\nonstopmode
\numberwithin{equation}{section}
\theoremstyle{plain}
\newtheorem{theorem}[equation]{Theorem}
\newtheorem{conjecture}[equation]{Conjecture}
\newtheorem{lemma}[equation]{Lemma}
\newtheorem{corollary}[equation]{Corollary}

\newtheorem{proposition}[equation]{Proposition}

\theoremstyle{definition}
\newtheorem{definition}[equation]{Definition}
\newtheorem{example}[equation]{Example}
\newtheorem{remark}[equation]{Remark}
\newtheorem{nonsec}[equation]{}

\theoremstyle{remark}

\newcommand{\R}{\mathbb{R}}

\newcommand{\C}{\mathbb{C}}

\newcommand{\B}{\mathbb{B}}
\newcommand{\M}{\mathsf{M}}
\newcommand{\uhp}{\mathbb{H}}

\font\fFt=eusm10 
\font\fFa=eusm7  
\font\fFp=eusm5  
\def\K{\mathchoice
{\hbox{\,\fFt K}}
{\hbox{\,\fFt K}}
{\hbox{\,\fFa K}}
{\hbox{\,\fFp K}}}

{\qed\bigskip}

\newcounter{alphabet}

\newcounter{minutes}
\divide\time by 60
\newcounter{hours}
\multiply\time by 60
\addtocounter{minutes}{-\time}

\begin{document}
\bibliographystyle{amsplain}
\title
{
M\"obius metric in sector domains
}

\def\thefootnote{}
\footnotetext{
\texttt{\tiny File:~\jobname .tex,
          printed: \number\year-\number\month-\number\day,
          \thehours.\ifnum\theminutes<10{0}\fi\theminutes}
}
\makeatletter\def\thefootnote{\@arabic\c@footnote}\makeatother

\author[O. Rainio]{Oona Rainio}
\author[M. Vuorinen]{Matti Vuorinen}

\begin{abstract}
The M\"obius metric $\delta_G$ is studied in the cases where its domain $G$ is an open sector of the complex plane. We introduce upper and lower bounds for this metric in terms of the hyperbolic metric and the angle of the sector, and then use these results to find bounds for the distortion of the M\"obius metric under quasiregular mappings defined in sector domains. Furthermore, we numerically study the M\"obius metric and its connection to the hyperbolic metric in polygon domains.
\end{abstract}
\maketitle

\textbf{Author information.}

\noindent Oona Rainio$^{(1)}$ email: \texttt{ormrai@utu.fi} ORCID: 0000-0002-7775-7656\\
Matti Vuorinen$^{(1)}$ email: \texttt{vuorinen@utu.fi} ORCID: 0000-0002-1734-8228\\
(1): Department of Mathematics and Statistics, University of Turku, FI-20014 Turku, Finland

\textbf{Funding.} Oona Rainio's research was funded by the University of Turku Graduate School UTUGS.

\textbf{Data availability statement.} Not applicable, no new data was generated.

\textbf{Conflict of interest statement.} No conflict of interest.

\section{Introduction}

One of the most important concepts in the geometric function theory is the \emph{intrinsic distance}. It means that, given two points in a domain, we do not only consider how close these points are to each other but also how they are located with respect to the boundary of the domain. In order to measure these kinds of distances, we need to use suitable intrinsic or hyperbolic type metrics, which have been recently studied, for instance, in \cite{chkv, hkvbook, h18, hvz, himps, fss, inm, sqm}.

In this article, we focus on one of these intrinsic metrics, which is defined as follows: For any domain $G\subset\overline{\R}^n=\R^n\cup\{\infty\}$ whose complement $(\overline{\R}^n\backslash G)$ contains at least two points, let the \emph{M\"obius metric} be the function $\delta_G:G\times G\to[0,\infty),$
\begin{align}\label{def_delta}
\delta_G(x,y)=\sup_{a,b\in\partial G}\log(1+|a,x,b,y|),    
\end{align}
where $|a,x,b,y|$ is the cross-ratio defined in \eqref{def_crossratio}. This metric was first introduced in \cite[pp. 115-116]{cgqm} and then later studied more extensively by P. Seittenranta in his PhD thesis \cite[Def. 1.1, p. 511]{S99}, which is why it is sometimes also referred to as Seittenranta's metric.

Due to the M\"obius invariance of the cross-ratio, the distances defined with the M\"obius metric are preserved under M\"obius transformations, which is one of the most useful properties of this metric. However, there are still numerous open questions concerning this metric. For instance, while it is known that the value of this metric is equal to that of the hyperbolic metric $\rho$ or the distance ratio metric $j$ in some special cases, see Theorems \ref{thm_seitminvariant} and \ref{thm_deltaj}, the M\"obius metric is studied very little in other kinds of domains. To fill this gap, our aim here is to find more information about the M\"obius metric in the cases where the domain $G$ is either an open sector of the complex plane or a polygon.

The main result of this article is as follows:

\begin{theorem}\label{thm_finalfordeltaS}
For all points $x,y$ in an open sector $S_\theta$ with an angle $0<\theta<2\pi$, the following inequalities hold:
\begin{align*}
(1)\quad&\rho_{S_\theta}(x,y)\leq\delta_{S_\theta}(x,y)\leq\min\left\{2,\left(\frac{\pi\sin(\theta\slash2)}{\theta}\right)^2\right\}\rho_{S_\theta}(x,y),\quad\text{if}\quad\theta<\pi,\\
(2)\quad&\delta_{S_\theta}(x,y)=\rho_{S_\theta}(x,y),\quad\text{if}\quad\theta=\pi,\\
(3)\quad&\max\left\{2{\rm arth}\frac{{\rm th}(\rho_{S_\theta}(x,y)\slash2)}{2},
\left(\frac{\pi\sin(\theta\slash2)}{\theta}\right)^2\rho_{S_\theta}(x,y)\right\}
\leq\delta_{S_\theta}(x,y)
\leq4\psi,\quad\text{if}\quad\theta>\pi,\\
&\text{where}\\
&\psi=
\begin{cases}
\min\left\{\rho_{S_\theta}(x,y),{\rm arth}\left((\theta\slash\pi){\rm th}(\rho_{S_\theta}(x,y)\slash2)\right)\right\},
\quad\text{if}\quad(\theta\slash\pi){\rm th}(\rho_{S_\theta}(x,y)\slash2)<1,\\
\rho_{S_\theta}(x,y)\quad\text{otherwise}.
\end{cases}
\end{align*}
\end{theorem}

The structure of this article is as follows. First, in Section \ref{sct3}, we combine some already known inequalities to create some initial bounds for the M\"obius metric in a general domain. Then, in Section \ref{sct4}, we study the M\"obius metric defined in an open sector by showing how the supremum of the cross-ratio in its definition \eqref{def_delta} can be found. These results will be used in Section \ref{sct5}, where we introduce bounds for the M\"obius metric in terms of the hyperbolic metric in a sector and prove Theorem \ref{thm_finalfordeltaS}. In Section \ref{sct6}, we 
apply these results and prove bounds for the distortion of the M\"obius metric under quasiregular mappings of the unit disk into sector domains. Finally, in Section \ref{sct7}, we  utilise the recent
computational methods from \cite{nrv1} to experimentally study the inequalities between the M\"obius and hyperbolic metric in polygon domains and formulate a few conjectures.
\section{Preliminaries}

First, define the following notations for the Euclidean metric. Let the distance from a point $x\in\R^n$ to a non-empty set $F\subset\R^n$ be $d(x,F)=\inf\{|x-z|\text{ }|\text{ }z\in F\}$. For a domain $G\subsetneq\R^n$, denote $d_G(x)=d(x,\partial G)$ for all $x\in G$. Let the Euclidean diameter of a non-empty set $F$ be $d(F)$ and the Euclidean distance between two non-empty separate sets $F_0,F_1$ be $d(F_0,F_1)$. Furthermore, denote the Euclidean open ball with a center $x\in\R^n$ and a radius $r>0$ by $B^n(x,r)$, the corresponding closed ball by $\overline{B}^n(x,r)$ and its boundary sphere by $S^{n-1}(x,r)$.

Let $\overline{\R}^n=\R^n\cup\{\infty\}$ be as in the introduction, and also denote $\overline{\C}^n=\C^n\cup\{\infty\}$. For all distinct points $x,y\in\overline{\R}^n$, define the \emph{spherical (chordal) metric} as in \cite[(3.6), p. 29]{hkvbook}:
\begin{align*}
q(x,y)&=\frac{|x-y|}{\sqrt{1+|x|^2}\sqrt{1+|y|^2}},\quad\text{if}\quad x\neq\infty\neq y;\quad
q(x,\infty)=\frac{1}{\sqrt{1+|x|^2}}.
\end{align*}
For any four distinct points $a,b,c,d\in\overline{\R}^n$, define the \emph{cross-ratio} as \cite[(3.10), p. 33]{hkvbook}
\begin{align}\label{def_crossratio}
|a,b,c,d|=\frac{q(a,c)q(b,d)}{q(a,b)q(c,d)},
\end{align}
and note that, if $\infty\notin\{a,b,c,d\}$, then this definition yields
\begin{align*}
|a,b,c,d|=\frac{|a-c||b-d|}{|a-b||c-d|}.
\end{align*}

Other than the M\"obius metric, we will be needing a few other hyperbolic type metrics. Define the upper half-space $\uhp^n=\{(x_1,...,x_n)\in\R^n\text{ }|\text{ }x_n>0\}$, the unit ball $\B^n=B^n(0,1)$ and the open sector $S_\theta=\{x\in\C\backslash\{0\}\text{ }|\text{ }0<\arg(x)<\theta\}$ with an angle $\theta\in(0,2\pi)$. Here, $\arg(x)\in[0,2\pi)$ denotes the principal branch of the argument of a complex number $x\in\C\backslash\{0\}$. Now, we can define the \emph{hyperbolic metric} in these three domains by using the following formulas, respectively \cite[(4.8), p. 52 \& (4.14), p. 55]{hkvbook}:
\begin{align*}
\text{ch}\rho_{\uhp^n}(x,y)&=1+\frac{|x-y|^2}{2d_{\uhp^n}(x)d_{\uhp^n}(y)},\quad x,y\in\uhp^n,\\
\text{sh}^2\frac{\rho_{\B^n}(x,y)}{2}&=\frac{|x-y|^2}{(1-|x|^2)(1-|y|^2)},\quad x,y\in\B^n,\\
\rho_{S_\theta}(x,y)&=\rho_{\uhp^2}(x^{\pi\slash\theta},y^{\pi\slash\theta}),\quad x,y\in S_\theta,
\end{align*}
From these formulas, it follows that:
\begin{align}
\text{th}\frac{\rho_{\uhp^2}(x,y)}{2}=\left|\frac{x-y}{x-\overline{y}}\right|,\quad
\text{th}\frac{\rho_{\B^2}(x,y)}{2}&=\left|\frac{x-y}{1-x\overline{y}}\right|,
\end{align}
where $\overline{y}$ is the complex conjugate of $y$.

For a domain $G\subsetneq\R^n$, the \emph{distance ratio metric} \cite[p. 25]{cgqm} introduced by Gehring and Osgood \cite{GO79} is the function $j_G:G\times G\to[0,\infty)$,
\begin{align*}
j_G(x,y)=\log(1+\frac{|x-y|}{\min\{d_G(x),d_G(y)\}}).   
\end{align*}
As noted in \cite[2.2, p. 1123 \& Lemma 2.1, p. 1124]{hvz}, this metric can be used to define another metric, the so called \emph{$j^*_G$-metric} $j^*_G:G\times G\to[0,1],$
\begin{align*}
j^*_G(x,y)={\rm th}\frac{j_G(x,y)}{2}=\frac{|x-y|}{|x-y|+2\min\{d_G(x),d_G(y)\}}.    
\end{align*}
Furthermore, the \emph{quasihyperbolic metric} introduced by Gehring and Palka in \cite{gp} is defined as the function $k_G:G\times G\to[0,\infty)$
\begin{align*}
k_G(x,y)=\inf_{\gamma\in\Gamma_{xy}}\int_\gamma\frac{|dx|}{d_G(x)},    
\end{align*}
where $\Gamma_{xy}$ consists of all the rectifiable curves in $G$ joining $x$ and $y$. Consider yet the \emph{triangular ratio metric} \cite[(1.1), p. 683]{chkv} $s_G:G\times G\to[0,1],$ 
\begin{align*}
s_G(x,y)=\frac{|x-y|}{\inf_{z\in\partial G}(|x-z|+|z-y|)}, 
\end{align*}
which was originally introduced by P. H\"ast\"o in 2002 \cite{h}.

The following result expresses the main property of the M\"obius metric:

\begin{theorem}\label{thm_seitminvariant} \cite{S99}, \cite[Thm 5.16, p. 75]{hkvbook}
The M\"obius metric $\delta_G$ is M\"obius invariant: If $G\subset\overline{\R}^n$ is a domain such that $\overline{\R}^n\backslash G$ contains at least two points and $h:\overline{\R}^n\to\overline{\R}^n$ is a M\"obius transformation, then for all $x,y\in G$,
\begin{align*}
\delta_{h(G)}(h(x),h(y))=\delta_G(x,y).    
\end{align*}
Furthermore, $\delta_G=\rho_G$ for $G\in\{\B^n,\uhp^n\}$.
\end{theorem}
\section{General inequalities}\label{sct3}

In this section, we will briefly review a few already existing inequalities and show how they can be used to create bounds for the M\"obius metric. Note that the inequalities found here concern mostly the situation, in which the shape of the domain $G$ is not known. For instance, Corollary \ref{cor_deltauni} gives us an inequality for a simply connected uniform domain $G$, but its constants are probably not very sharp when compared to those that could be obtained when knowing the exact shape of the domain $G$.

\begin{theorem}\label{thm_deltaj}\cite[Thm 5.16, p. 75]{hkvbook}
For all points $x,y$ in a domain $G\subsetneq\R^n$,
\begin{align*}
j_G(x,y)\leq\delta_G(x,y)\leq 2j_G(x,y)    
\end{align*}
and, in the special case $G=\R^n\backslash\{0\}$, $\delta_G=j_G$. 
\end{theorem}

\begin{theorem}\label{thm_krho}\cite[(3.2.1), p. 35]{gh}
For all points $x,y$ in a simply connected domain $G\subsetneq\R^2$,
\begin{align*}
\frac{1}{2}k_G(x,y)\leq\rho_G(x,y)\leq2k_G(x,y).    
\end{align*}
\end{theorem}

\begin{theorem}\label{thm_jk}\cite[Cor. 5.6, p. 69]{hkvbook}
For all points $x,y$ in a domain $G\subsetneq\R^n$, $j_G(x,y)\leq k_G(x,y)$.
\end{theorem}

\begin{definition}\label{def_uni}\cite[Def. 6.1, p. 84]{hkvbook},\cite[Def. 2.4, p. 8]{linden}
A domain $G\subset\R^n$ is \emph{uniform} if there exists a number $A\geq1$ such that the inequality $k_G(x,y)\leq Aj_G(x,y)$ holds for all $x,y\in G$ and the smallest such number $A$ fulfilling this condition is called the \emph{uniformity constant} of $G$.
\end{definition}

\begin{corollary}\label{cor_deltauni}
If a domain $G\subsetneq\R^2$ is simply connected and uniform with the uniformity constant $A_G$, then
\begin{align*}
\rho_G(x,y)\slash(2A_G)\leq\delta_G(x,y)\leq4\rho_G(x,y)    
\end{align*}
for all $x,y\in G$.
\end{corollary}
\begin{proof}
Follows from Theorems \ref{thm_jk} and \ref{thm_krho}, and Definition \ref{def_uni}.
\end{proof}

Now, let us find some bounds for the M\"obius metric with the triangular ratio metric and the $j^*$-metric in the cases of both a convex domain $G$ and a non-convex one.

\begin{lemma}\label{lem_sj}\cite[Lemma 2.1, p. 1124; Lemma 2.2, p. 1125 \& Thm 2.9(i), p. 1129]{hvz}
For all points $x,y$ in a domain $G\subsetneq\R^n$,
\begin{align*}
j^*_G(x,y)\leq s_G(x,y)\leq 2j^*_G(x,y)    
\end{align*}
and, if $G$ is convex, the constant 2 above can be replaced by $\sqrt{2}$.
\end{lemma}

\begin{lemma}\label{lem_sthj}\cite[Lemma 2.7(ii), p. 1128]{hvz}
For all points $x,y$ in a convex domain $G\subsetneq\R^n$,
\begin{align*}
{\rm th}\frac{j_G(x,y)}{2}\leq s_G(x,y)\leq{\rm th}j_G(x,y).
\end{align*}
\end{lemma}

\begin{corollary}\label{cor_deltasj}
For a domain $G\subsetneq\R^n$ such that $\overline{\R}^n\backslash G$ contains at least two points and for all $x,y\in G$, the following inequalities hold:\newline
$(1)$ $j^*_G(x,y)\leq{\rm th}(\delta_G(x,y)\slash2)\leq{\rm th}(2{\rm arth}(j^*_G(x,y)))\leq2j^*_G(x,y)$,\newline
$(2)$ $s_G(x,y)\slash2\leq{\rm th}(\delta_G(x,y)\slash2)\leq{\rm th}(2{\rm arth}(s_G(x,y)))\leq2s_G(x,y)$.\newline
Furthermore, if $G$ is convex, then for all $x,y\in G$\newline 
$(3)$ $s_G(x,y)\slash\sqrt{2}\leq{\rm th}(\delta_G(x,y)\slash2)$,\newline
$(4)$ $s_G(x,y)\leq{\rm th}(\delta_G(x,y))$.
\end{corollary}
\begin{proof}
(1) Follows from Theorem \ref{thm_deltaj} and the definition of $j^*$-metric.\newline
(2) Follows from the first inequality and Lemma \ref{lem_sj}.\newline
(3) Follows from the first inequality and Lemma \ref{lem_sj}.\newline
(4) Follows from Theorem \ref{thm_deltaj} and Lemma \ref{lem_sthj}.
\end{proof}

Finally, let us consider the case where the domain $G$ is an open sector. Note that neither the inequalities of Corollary \ref{cor_rhodeltaWithS} nor Corollary \ref{cor_rhodeltaSuni} have the best possible constants. In fact, they are only used to prove our main result, Theorem \ref{thm_finalfordeltaS}, in Section \ref{sct5}.

\begin{theorem}\label{thm_rhos}\cite[Cor. 4.9, p. 9]{sqm}
For a fixed angle $\theta\in(0,2\pi)$ and for all $x,y\in S_\theta$, the following results hold:\newline 
(1) $s_{S_\theta}(x,y)\leq
{\rm th}(\rho_{S_\theta}(x,y)\slash2)\leq
(\pi\slash\theta)\sin(\theta\slash2) s_{S_\theta}(x,y)$ if $\theta\in(0,\pi)$,\newline
(2) $s_{S_\theta}(x,y)={\rm th}(\rho_{S_\theta}(x,y)\slash2)$ if $\theta=\pi$,\newline
(3) $(\pi\slash\theta)s_{S_\theta}(x,y)\leq
{\rm th}(\rho_{S_\theta}(x,y)\slash2)\leq
 s_{S_\theta}(x,y)$ if $\theta\in(\pi,2\pi)$.\newline
 Furthermore, these bounds are also sharp.
\end{theorem}

\begin{corollary}\label{cor_rhodeltaWithS}
For all points $x,y\in S_\theta$, the following inequalities hold:
\begin{align*}
(1)\quad&\max\left\{2{\rm arth}\left(\frac{\theta}{\sqrt{2}\pi\sin(\theta\slash2)}{\rm th}\frac{\rho_{S_\theta}(x,y)}{2}\right),{\rm arth}\left(\frac{\theta}{\pi\sin(\theta\slash2)}{\rm th}\frac{\rho_{S_\theta}(x,y)}{2}\right)\right\}\\
&\leq\delta_{S_\theta}(x,y)
\leq2\rho_{S_\theta}(x,y),
\quad\text{if}\quad0<\theta<\pi,
\\
(2)\quad&2{\rm arth}\frac{{\rm th}(\rho_{S_\theta}(x,y)\slash2)}{2}
\leq\delta_{S_\theta}(x,y),
\quad\text{if}\quad\pi<\theta<2\pi,\\
(3)\quad&\delta_{S_\theta}(x,y)
\leq4{\rm arth}\left(\frac{\theta}{\pi}{\rm th}\frac{\rho_{S_\theta}(x,y)}{2}\right),
\quad\text{if}\quad\pi<\theta<2\pi\quad\text{and}\quad\frac{\theta}{\pi}{\rm th}\frac{\rho_{S_\theta}(x,y)}{2}<1.
\end{align*}
\end{corollary}
\begin{proof}
Follows from Corollary \ref{cor_deltasj} and Theorem \ref{thm_rhos}.
\end{proof}

\begin{theorem}\label{thm_unisForS}\cite[Thm 1.7 \& Thm 1.8, p. 6]{linden}
An open sector $S_\theta$ is uniform with the constant $A_\theta$ that fulfills
\begin{align*}
&A_\theta=\frac{1}{\sin(\theta\slash2)}+1,\quad\text{if}\quad0<\theta\leq\pi,\quad\text{and}\\
&\max\left\{2,\frac{2\log(\tan(\theta\slash4))+\theta-\pi}{\log(1-2\cos(\theta\slash2))}\right\}\leq A_\theta\leq4\left(\frac{\theta}{2\pi-\theta}\right)^2\left(\frac{1}{\sin(\theta\slash2)}+1\right),\\
&\text{if}\quad\pi<\theta<2\pi.
\end{align*}
\end{theorem}

\begin{corollary}\label{cor_rhodeltaSuni}
For all points $x,y\in S_\theta$,
\begin{align*}
&(1)\quad\frac{\sin(\theta\slash2)}{2(1+\sin(\theta\slash2))}\rho_{S_\theta}(x,y)\leq\delta_{S_\theta}(x,y)\leq4\rho_{S_\theta}(x,y),
\quad\text{if}\quad0<\theta\leq\pi,\\
&(2)\quad\frac{1}{8}\left(\frac{2\pi}{\theta}-1\right)^2\frac{\sin(\theta\slash2)}{1+\sin(\theta\slash2)}\rho_{S_\theta}(x,y)\leq\delta_{S_\theta}(x,y)\leq4\rho_{S_\theta}(x,y)
\quad\text{if}\quad\pi<\theta<2\pi.
\end{align*}
\end{corollary}
\begin{proof}
Follows from Corollary \ref{cor_deltauni} and Theorem \ref{thm_unisForS}.
\end{proof}
\section{M\"obius metric in open sector}\label{sct4}

In this section, our aim is to find ways to estimate the value of the M\"obius metric defined in an open sector $S_\theta$. To do this, we will study the supremum of the cross-ratio needed in the definition of the metric $\delta_{S_\theta}$, in both cases where the angle $\theta$ is less than $\pi$ and greater than $\pi$. The main result of this section is Corollary \ref{cor_deltaInS} but, in order to prove it, we need to consider several other results first.

\begin{proposition}\label{prop_abinSH}
$(1)$ If $x,y\in\uhp^2$ such that $|x|=|y|=r>0$ and $\arg(x)\leq\arg(y)$, then $\sup_{a,b\in \R}|a,x,b,y|=|r,x,-r,y|$.\newline
$(2)$ If $x,y\in i\R\cap\B^2$ such that ${\rm Im}(x)\leq{\rm Im}(y)$, then $\sup_{a,b\in S^1}|a,x,b,y|=|-i,x,i,y|$.
\end{proposition}
\begin{proof}
Since $\delta_{\uhp^2}(x,y)=\sup_{a,b\in\R}\log(1+|a,x,b,y|)$, both results can be verified by the fact that $\delta_G=\rho_G$ for $G\in\{\B^n,\uhp^n\}$ according to Theorem \ref{thm_seitminvariant}.
\end{proof}

\begin{lemma}\label{lem_smallsectorontolens}
For all points $x,y$ in an open sector $S_\theta$ with an angle $0<\theta<\pi$ such that $\arg(x)\leq\arg(y)$ and $|x|=|y|=r>0$, there is a M\"obius transformation $f$ that maps $S_\theta$ onto the lens-shaped domain
\begin{align*}
B^2((1-u)i,u)\cap B^2((u-1)i,u),\quad u=\frac{1}{1-\cos(\theta\slash2)}>1, \end{align*}
and $x,y$ into $f(x),f(y)\in i\R\cap\B^2$ so that ${\rm Im}(f(x))\leq{\rm Im}(f(y))$.
\end{lemma}
\begin{proof}
Define the function $f:\overline{\C}\to\overline{\C}$, 
\begin{align}\label{functionF}
f(z)=\frac{-i(1+e^{\theta i\slash2})(z-re^{\theta i\slash2})}{(1-e^{\theta i\slash2})(z+re^{\theta i\slash2})}.     
\end{align}
Clearly, $f$ is the M\"obius transformation that fulfills
\begin{align}\label{mappingsOfF}
f(r)=-i,\quad
f(re^{\theta i\slash2})=0,\quad
f(re^{\theta i})=i,\quad
f(0)=\frac{-\sin(\theta\slash2)}{1-\cos(\theta\slash2)}=-f(\infty).
\end{align}
By the general properties of M\"obius transformations, $f$ preserves the angles and must turn a line into a circle if any three points chosen from it are no longer collinear after the transformation. Thus, $f$ maps two sides of the sector $S_\theta$ onto two circular arcs that are symmetric with respect to the both coordinate axes, meet each other in the points $f(0)$ and $f(\infty)$ at an angle of $\theta$, and out of which one contains the point $i$ and the other one the point $-i$. See Figure \ref{fig1}. 

Consider a circle $S^1((1-u)i,u)$, $u>1$. Clearly, it is symmetric with respect to the coordinate axes and $i\in S^1((1-u)i,u)$. Using simple trigonometry, it can be calculated that the two interior angles of the figure consisting of the real axis and the circular arc $S^1((1-u)i,u)\cap\uhp^2$ are
\begin{align*}
\frac{\pi}{2}-\arcsin\left(\frac{u-1}{u}\right)\in\left(0,\frac{\pi}{2}\right).    
\end{align*}
If the value of this angle is $\theta\slash2$, we can solve that
\begin{align}\label{ucos}
u=\frac{1}{1-\cos(\theta\slash2)}.   
\end{align}

By combining all our observations made above, we will have that, for all $0<\theta<\pi$, 
\begin{align*}
f(\partial S_\theta)&=(S^1((1-u)i,u)\cap\uhp^2)\cup(S^1((u-1)i,u)\backslash\uhp^2),\\
f(S_\theta)&=B^2((1-u)i,u)\cap B^2((u-1)i,u),
\end{align*}
where $u$ is as in \eqref{ucos}. The final part of the lemma is very trivial: From the behaviour of the points in \eqref{mappingsOfF}, we see that the transformation $f$ maps the circle $S^1(0,r)$ onto the imaginary axis and, if $x,y\in S^1(0,r)$ such that $0<\arg(x)\leq\arg(y)<\theta$, then clearly $f(x),f(y)\in[-i,i]$ so that $-1<{\rm Im}(f(x))\leq{\rm Im}(f(y))<1$.
\end{proof}

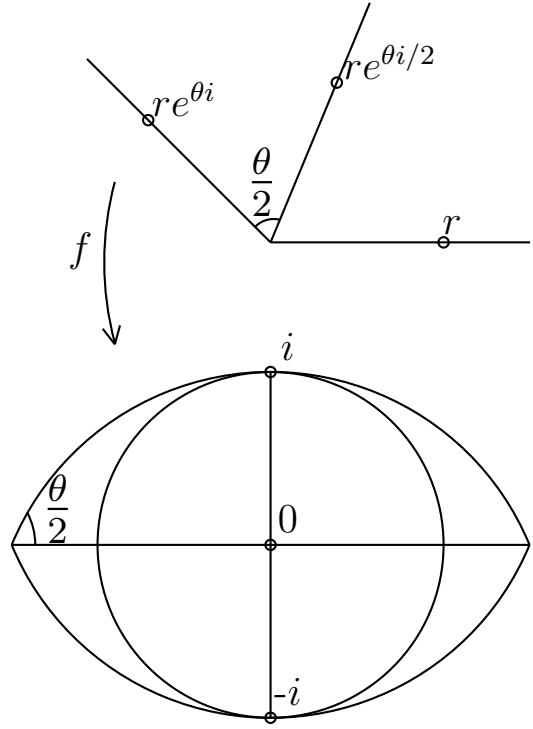
\begin{figure}[ht]
    \centering
    \begin{tikzpicture}[scale=2.3]
    \draw[thick] (0,1.75) -- (-1.061,2.811);
    \draw[thick] (0,1.75) -- (1.5,1.75);
    \draw[thick] (0,1.75) -- (0.574,3.136);
    \draw[thick] (1,1.75) circle (0.03cm);
    \node[scale=1.3] at (1.05,1.85) {$r$};
    \draw[thick] (0.383,2.674) circle (0.03cm);
    \node[scale=1.3] at (0.69,2.8) {$re^{\theta i\slash2}$};
    \draw[thick] (-0.707,2.457) circle (0.03cm);
    \node[scale=1.3] at (-0.5,2.58) {$re^{\theta i}$};
    \draw[thick] (0.057,1.88) arc (75:137:0.15);
    \node[scale=1.8] at (-0.05,2.1) {$\frac{\theta}{2}$};
    \draw[thick] (-0.9,2.1) arc (165:195:1.8);
    \draw[thick] (-0.9,1.16) -- (-0.87,1.27);
    \draw[thick] (-0.9,1.16) -- (-0.97,1.25);
    \node[scale=1.3] at (-1.1,1.7) {$f$};
    \draw[thick] (0,-1) -- (0,1);
    \draw[thick] (-1.497,0) -- (1.497,0);
    \draw[thick] (0,0) circle (1cm);
    \draw[thick] (1.497,0) arc (22.5:157.5:1.620);
    \draw[thick] (1.497,0) arc (-22.5:-157.5:1.620);
    \draw[thick] (0,1) circle (0.03cm);
    \node[scale=1.3] at (0.1,1.15) {$i$};
    \draw[thick] (0,0) circle (0.03cm);
    \node[scale=1.3] at (0.1,0.15) {$0$};
    \draw[thick] (0,-1) circle (0.03cm);
    \node[scale=1.3] at (0.1,-0.85) {-$i$};
    \draw[thick] (-1.36,0) arc (0:28:0.4);
    \node[scale=1.8] at (-1.23,0.21) {$\frac{\theta}{2}$};
    \end{tikzpicture}
    \caption{Sector $S_\theta$ before and after the M\"obius transformation $f$ defined in \eqref{functionF}, when $\theta=3\pi\slash4$.}
    \label{fig1}
\end{figure}

The result of Lemma \ref{lem_smallsectorontolens} is very useful because it follows from the M\"obius invariance of the cross-ratio that the value of the M\"obius metric between $x,y\in S_\theta$ can be calculated in the lens-shaped symmetric domain $f(S_\theta)$ for $f(x),f(y)$,
see Figure \ref{fig1}.

\begin{theorem}\label{thm_deltaPointslessthanpi}
For all $0<\theta<\pi$ and $x,y\in S_\theta$ such that $\arg(x)\leq\arg(y)$ and $|x|=|y|=r>0$, the supremum $\sup_{a,b\in\partial S_\theta}|a,x,b,y|$ is given by the points $a=r$ and $b=re^{\theta i}$.
\end{theorem}
\begin{proof}
Let $f$ be the M\"obius transformation defined in \eqref{functionF}, under which the open sector $S_\theta$ with an angle $0<\theta<\pi$ is mapped onto a lens-shaped domain $f(S_\theta)$. For all points $x,y\in i\R\cap\B^2$, choose $a,b\in f(\partial S_\theta)$ so that the cross-ratio $|a,x,b,y|$ is at greatest. Note that, for all points $u\in i\R\cap\B^2$ and $v\in\C\backslash\B^2$, the inequality
\begin{align}\label{ine_uiv}
\max\left\{\frac{|i-v|}{|u-i|},\frac{|-i-v|}{|u-(-i)|}\right\}\geq1,
\end{align}
holds. It follows from this that $|x-a|\leq|a-b|$ holds, because otherwise replacing $a$ either by $i$ or $-i$ would give a greater value for the cross-ratio $|a,x,b,y|$.

Fix now $a'\in[x,a]\cap S^1$. By the inequality $|x-a|\leq|a-b|$ and the triangle inequality,
\begin{align}\label{ine_x1}
\frac{|a-b|}{|x-a|}
\leq\frac{|a-b|-|a-a'|}{|x-a|-|a-a'|}
=\frac{|a-b|-|a-a'|}{|x-a'|}
\leq\frac{|a'-b|}{|x-a'|}.
\end{align}
Let us yet show that there is a point $b'\in S^1$ such that
\begin{align}\label{ine_y1}
\frac{|a'-b|}{|y-b|}
\leq\frac{|a'-b'|}{|y-b'|}.
\end{align}
If $|y-b|\leq|a'-b|$ holds for the point $b'\in[y,b]\cap S^1$, then the inequality \eqref{ine_y1} follows from the triangle inequality just like \eqref{ine_x1}. If $|y-b|>|a'-b|$ instead, then there is $b\in\{i,-i\}$ such that $|y-b'|\leq|a'-b'|$ by the inequality \eqref{ine_uiv} and the inequality \eqref{ine_y1} clearly holds for this choice of $b'$. 

Thus, if $a,b\in f(\partial S_\theta)$ give the supremum of $|a,x,b,y|$ for given points $x,y\in i\R\cap\B^2$ and $a',b'$ are chosen like above, it follows from the inequalities \eqref{ine_x1} and \eqref{ine_y1} that
\begin{align*}
|a,x,b,y|
\leq|a',x,b,y|
\leq|a',x,b',y|
\leq\sup_{a',b'\in S^1}|a',x,b',y|.
\end{align*}
By Proposition \ref{prop_abinSH}(2), if $x,y\in i\R\cap\B^2$ such that $\text{Im}(x)\leq\text{Im}(y)$, then $\sup_{a',b'\in S^1}|a',x,b',y|$ is given by $a=-i$ and $b=i$. Since $i,-i\in f(\partial S_\theta)\cap S^1$, it must hold that
\begin{align*}
\sup_{a,b\in f(\partial S_\theta)}|a,x,b,y|=|-i,x,i,y|.
\end{align*}
Because $f$ preserves the cross-ratio as a M\"obius transformation, we can now show that, for all $x,y\in S_\theta$ such that $|x|=|y|=r$ and $\arg(x)\leq\arg(y)$,
\begin{align*}
\sup_{a,b\in\partial S_\theta}|a,x,b,y|
=\sup_{a,b\in f(\partial S_\theta)}|a,f(x),b,f(y)|
=|-i,f(x),i,f(y)|
=|r,x,re^{\theta i},y|.
\end{align*}
\end{proof}

Note that Theorem \ref{thm_deltaPointslessthanpi} does not hold in the case $\theta>\pi$, as the following example shows.

\begin{example}\label{ex_crossratiooverpi}
For $x=e^{(1-k)\theta i\slash2}$ and $y=e^{(1+k)\theta i\slash2}$ with $0<k<1$ and $\pi<\theta<2\pi$,
\begin{align*}
\lim_{k\to0^+}\frac{|1,x,e^{\theta i},y|}{|0,x,\infty,y|} 
=\lim_{k\to0^+}\frac{\sin(\theta\slash2)}{2\sin^2((1-k)\theta\slash4)}
=\frac{\sin(\theta\slash2)}{2\sin^2(\theta\slash4)}
=\frac{\cos(\theta\slash4)}{\sin(\theta\slash4)}<1
\end{align*}
and it follows that $\sup_{a,b\in\partial S_\theta}|a,x,b,y|$ is not attained with $a=1$ and $b=e^{\theta i}$.
\end{example}

However, we can still use Lemma \ref{lem_smallsectorontolens} to calculate the supremum of the cross-ratio in the M\"obius metric for points $x,y$ in a sector $S_\theta$ with $\theta>\pi$, as can be seen from the following result.

\begin{figure}[ht]
    \centering
    \begin{tikzpicture}[scale=2.3]
    \draw[thick] (0,1.75) -- (1.5,1.75);
    \draw[thick] (1,1.75) circle (0.03cm);
    \node[scale=1.3] at (1.08,1.85) {$r$};
    \draw[thick] (0,1.75) -- (-1.061,0.689);
    \draw[thick] (-0.707,1.043) circle (0.03cm);
    \node[scale=1.3] at (-0.707,1.23) {$re^{\theta i}$};
    \draw[thick] (0,1.75) -- (-0.574,3.136);
    \draw[thick] (-0.383,2.674) circle (0.03cm);
    \node[scale=1.3] at (-0.05,2.8) {$re^{\theta i\slash2}$};
    \draw[thick] (1,2.5) arc (115:65:1.3);
    \draw[thick] (2.09,2.5) -- (2,2.45);
    \draw[thick] (2.09,2.5) -- (2.05,2.61);
    \node[scale=1.3] at (1.6,2.8) {$f$};
    \draw[thick] (0.668+3,1.75) arc (-22:202:0.723);
    \draw[thick] (-0.668+3,1.75) arc (158:382:0.723);
    \draw[thick] (-0.668+3,1.75) -- (0.668+3,1.75);
    \draw[thick] (3,2.74) circle (0.03cm);
    \node[scale=1.3] at (3.1,2.86) {$i$};
    \draw[thick] (3,1.75) circle (0.03cm);
    \node[scale=1.3] at (3.1,1.9) {$0$};
    \draw[thick] (3,-1+2.75) circle (0.03cm);
    \node[scale=1.3] at (0.1,-0.85+2.75) {$0$};
    \draw[thick] (3,-2+2.75) circle (0.03cm);
    \node[scale=1.3] at (3.1,-1.85+2.75) {-$i$};
    \draw[thick] (-0.09,1.94) arc (150:216:0.25);
    \node[scale=1.8] at (-0.23,1.85) {$\frac{\theta}{2}$};
    \draw[thick] (-0.568+3,1.75) arc (0:90:0.15);
    \node[scale=1.8] at (-0.57+3,-0.68+2.75) {$\frac{\theta}{2}$};
    \end{tikzpicture}
    \caption{Sector $S_\theta$ before and after the M\"obius transformation $f$ defined in \eqref{functionF},
    when $\theta=5\pi\slash4$.}
    \label{figN}
\end{figure}

\begin{corollary}
For any open sector $S_\theta$ with an angle $\pi<\theta<2\pi$, there is a M\"obius transformation $f$ that maps $S_\theta$ onto the domain
\begin{align*}
B^2((1-u)i,u)\cup B^2((u-1)i,u),\quad u=\frac{1}{1-\cos(\theta\slash2)}\in\left(\frac{1}{2},1\right), \end{align*}
and, for all $x,y\in S_\theta$,
\begin{align*}
\sup_{a,b\in\partial S_\theta}|a,x,b,y|
=\sup_{a,b\in f(\partial S_\theta)}|a,f(x),b,f(y)|.    
\end{align*}
\end{corollary}
\begin{proof}
Let the M\"obius transformation $f$ be as in \eqref{functionF} with, for instance, $r=1$. The proof now goes just like that of Theorem \ref{thm_deltaPointslessthanpi}, but it must be noted that $f$ maps the sides of $S_\theta$ onto circular arcs that meet each other at an angle $\theta>\pi$. Thus, $f(S_\theta)$ must be an union of two disks $B^2((1-u)i,u)$ and $B^2((u-1)i,u)$ instead of their intersection and $1\slash2<u<1$ now, see Figure \ref{figN}. The final part of the proof follows from the M\"obius invariance of the cross-ratio.
\end{proof}

Several computational experiments support the next conjecture.

\begin{conjecture}
If $\pi<\theta<2\pi$ and $x=re^{(1-k)\theta i\slash2}$ and $y=re^{(1+k)\theta i\slash2}$ with $r>0$ and $0<k<1$, then
\begin{align*}
\sup_{a,b\in\partial S_\theta}|a,x,b,y|=\max\{|r,x,re^{\theta i},y|,|0,x,\infty,y|\}. 
\end{align*}
\end{conjecture}

The results of this section about the supremum of the cross-ratio give us information about the values of the M\"obius metric $\delta_{S_\theta}$ defined in a sector domain.

\begin{corollary}\label{cor_deltaInS}
For all $x,y\in S_\theta$ with $0<\theta<2\pi$ such that $|x|=|y|$ and $\arg(x)\leq\arg(y)$,
\begin{align*}
\delta_{S_\theta}(x,y)\geq\log\left(1+\frac{\sin(\theta\slash2)\sin((\arg(y)-\arg(x))\slash2)}{\sin(\arg(x)\slash2)\sin((\theta-\arg(y))\slash2)}\right), 
\end{align*}
where the equality holds whenever $\theta\leq\pi$.
\end{corollary}
\begin{proof}
Let $x=re^{ui}$ and $y=re^{vi}$ with $0<u\leq v<\theta$. By Theorem \ref{thm_deltaPointslessthanpi} and Proposition \ref{prop_abinSH}(1), the supremum $\sup_{a,b\in\partial S_\theta}|a,x,y,b|$ is now found by choosing $a=r$ and $b=re^{\theta i}$ if $\theta\leq\pi$, and these choices of $a,b$ give a lower limit for the supremum if $\theta>\pi$. The result follows now directly from the definition of $\delta_{S_\theta}(x,y)$.
\end{proof}
\section{M\"obius metric and hyperbolic metric in open sector}\label{sct5}

In this section, we will study the connection between the M\"obius metric and the hyperbolic metric in an open sector $S_\theta$ with an angle $0<\theta<2\pi$. The main result of this section is Corollary \ref{cor_rhodeltainS}, which will be used to prove Theorem \ref{thm_finalfordeltaS}. However, in order to derive these results, we need to define the following quotient.

For all $0<k<1$ and $0<\theta<2\pi$, define
\begin{align}\label{def_quoQ}
Q(k,\theta)
\equiv
\left. \log\left(1+\frac{\sin(\theta\slash2)\sin(k\theta\slash2)}{\sin^2((1-k)\theta\slash4)}\right) 
\middle/ \log\left(1+\frac{\sin(k\pi\slash2)}{\sin^2((1-k)\pi\slash4)}\right)
\right.
.
\end{align}

The quotient above is very much needed here because it equals to the value of the quotient between the M\"obius metric and the hyperbolic metric in certain cases, as will be shown in the next lemma.

\begin{lemma}\label{lem_Qktheta}
For all $x,y\in S_\theta$ such that $x=re^{(1-k)\theta i\slash2}$ and $y=re^{(1+k)\theta i\slash2}$ with $r>0$ and $0<k<1$,
\begin{align*}
\frac{\delta_{S_\theta}(x,y)}{\rho_{S_\theta}(x,y)}=Q(k,\theta),\quad\text{if}\quad0<\theta<\pi;\quad\text{and}\quad
\frac{\delta_{S_\theta}(x,y)}{\rho_{S_\theta}(x,y)}\geq Q(k,\theta),\quad\text{if}\quad\pi<\theta<2\pi.
\end{align*}
\end{lemma}
\begin{proof}
Recall the trigonometric identities $\sin(u)=\cos(\pi\slash2-u)$ and $\cos(2v)=1-2\sin^2(v)$. It follows from these that
\begin{align*}
1-\sin(k\pi\slash2)
=1-\cos((1-k)\pi\slash2)
=2\sin^2((1-k)\pi\slash4).
\end{align*}
By using this formula, we will have 
\begin{align*}
&\rho_{S_\theta}(x,y)
=\rho_{S_\theta}(re^{(1-k)\theta i\slash2},re^{(1+k)\theta i\slash2}) =\rho_{\uhp^2}(r^{\pi\slash\theta}e^{(1-k)\pi i\slash2},r^{\pi\slash\theta}e^{(1+k)\pi i\slash2})\\
&=\log\left(\frac{|e^{(1-k)\pi i\slash2}-e^{-(1+k)\pi i\slash2}|+|e^{(1-k)\pi i\slash2}-e^{(1+k)\pi i\slash2}|}{|e^{(1-k)\pi i\slash2}-e^{-(1+k)\pi i\slash2}|-|e^{(1-k)\pi i\slash2}-e^{(1+k)\pi i\slash2}|}\right)
=\log\left(\frac{1+\sin(k\pi\slash2)}{1-\sin(k\pi\slash2)}\right)\\
&=\log\left(1+\frac{2\sin(k\pi\slash2)}{1-\sin(k\pi\slash2)}\right)
=\log\left(1+\frac{\sin(k\pi\slash2)}{\sin^2((1-k)\pi\slash4)}\right).
\end{align*}
Combining the expression above and Corollary \ref{cor_deltaInS}, our result follows.
\end{proof}

While the result of Lemma \ref{lem_Qktheta} only holds for distinct points $x,y\in S_\theta$ that are symmetric with respect to the angle bisector of the sector and fulfill $|x|=|y|$, Corollary \ref{cor_Qwlg} shows us why studying the quotient $Q(k,\theta)$ is useful also outside these restrictions.

\begin{lemma}\label{lem_g}\cite[Lemma 4.2, pp. 7-8]{sqm}
For given two distinct points $x\,, y \in \uhp^2$, there exists a M\"obius transformation
$g:\uhp^2\to\uhp^2$ such that $|g(x)|=|g(y)| =1$ and ${\rm Im}(g(x)) = {\rm Im}(g(y))\,.$
\begin{enumerate}
    \item If ${\rm Im}(x) = {\rm Im}(y)\,,$ then $g(z)=(z-a)/r$
    where $a=  {\rm Re}((x+y)/2)$ and $r =|x-a|\,.$
    \item If ${\rm Re}(x) = {\rm Re}(y) =a\,$ and $r=\sqrt{{\rm Im}(x)  {\rm Im}(y)}$, then $g$ is the M\"obius transformation fulfilling $g(a-r)=0$, $g(a)=1$ and $g(a+r) =\infty$.
    \item In the remaining case, the angle $\alpha=\measuredangle(L(x,y),\R)$ belongs to $(0,\pi\slash2)$. Let $S^1(c_1,r_1)$ and $S^1(c_2,r_2)$ be two circles centered at the real axis and orthogonal to each other, such that $x,y\in S^1(c_1,r_1)$ and $c_2=L(x,y)\cap\R$. Then $g$ is determined by $g(B^2(c_1,r_1)\cap\uhp^2)=\B^2\cap\uhp^2\,,$  $g(c_1-r_1)=-1$, $g(c_1+r_1)=1$ and $g(S^1(c_2,r_2)\cap\uhp^2)=\{yi\text{ }|\text{ }y>0\}$.
\end{enumerate}
\end{lemma}

\begin{lemma}\label{lem_confmappingtoH}\cite[Lemma 4.5, p. 8]{sqm}
For all distinct points $x,y\in S_\theta$ with $0<\theta<2\pi$, there is a conformal mapping $f:S_\theta\to S_\theta$ such that $f(x)=e^{(1-k)\theta i\slash2}$ and $f(y)=e^{(1+k)\theta i\slash2}$ for some $k\in(0,1)$. \end{lemma}
\begin{proof}
Consider a conformal map $h:S_\theta\to\uhp^2$, $h(z)=z^{\pi\slash\theta}$. Fix $g:\uhp^2\to\uhp^2$ as the M\"obius invariant map of Lemma \ref{lem_g} for the points $h(x),h(y)$. Define a conformal mapping $f=h^{-1}\circ g\circ h$. Since $|g(h(x))|=|g(h(y))|=1$ and ${\rm Im}(g(h(x)))={\rm Im}(g(h(y)))$, we can write $g(h(x))=e^{(1-k)\pi i\slash2}$ and $g(h(y))=e^{(1+k)\pi i\slash2}$ for some $0<k<1$. By using this $k$, we will have the points $f(x)=e^{(1-k)\theta i\slash2}$ and $f(y)=e^{(1+k)\theta i\slash2}$. 
\end{proof}

\begin{corollary}\label{cor_Qwlg}
For all $0<\theta<2\pi$ and distinct $x,y\in S_\theta$,
\begin{align*}
\inf_{0<k<1}Q(k,\theta)&\leq
\frac{\delta_{S_\theta}(x,y)}{\rho_{S_\theta}(x,y)}\leq
\sup_{0<k<1}Q(k,\theta),\quad\text{if}\quad0<\theta\leq\pi\\
\inf_{0<k<1}Q(k,\theta)&\leq
\frac{\delta_{S_\theta}(x,y)}{\rho_{S_\theta}(x,y)},\quad\text{if}\quad\pi<\theta<2\pi
\end{align*}
where $f(x)=e^{(1-k)\theta i\slash2}$ and $f(y)=e^{(1+k)\theta i\slash2}$.
\end{corollary}
\begin{proof}
Let the mappings $f,g,h$ be as in Lemma \ref{lem_g} and the proof of Lemma \ref{lem_confmappingtoH}. Note that, even though the mapping $f$ of does not necessarily preserve the distance $\delta_{S_\theta}(x,y)$, by Theorem \ref{thm_seitminvariant} and the conformal invariance of the hyperbolic metric,
\begin{align*}
&\delta_{\uhp^2}(h(x),h(y))=\delta_{\uhp^2}(g(h(x)),g(h(y)))=\rho_{\uhp^2}(g(h(x)),g(h(y)))=\rho_{\uhp^2}(h(x),h(y)),\\
&\rho_{S_\theta}(x,y)=\rho_{S_\theta}(f(x),f(y)),
\end{align*}
for all points $x,y\in S_\theta$. It follows from this that 
\begin{align*}
\inf_{x,y\in S_\theta}\frac{\delta_{S_\theta}(f(x),f(y))}{\rho_{S_\theta}(f(x),f(y))}\leq
\frac{\delta_{S_\theta}(x,y)}{\rho_{S_\theta}(x,y)}\leq
\sup_{x,y\in S_\theta}\frac{\delta_{S_\theta}(f(x),f(y))}{\rho_{S_\theta}(f(x),f(y))},
\end{align*}
which leads to the result of our corollary by Lemmas \ref{lem_Qktheta} and \ref{lem_confmappingtoH}. 
\end{proof}

Before studying the values of the quotient $Q(k,\theta)$, consider yet the following proposition.

\begin{proposition}\label{prop_triginc}
$(1)$ For all constants $u,v\in(0,\pi]$, the quotients $\sin(uk)\slash\sin(vk)$ and $\sin((1-k)v)\slash\sin((1-k)u)$ are increasing with respect to $0<k<1$ if and only if $u\leq v$, and decreasing if $u>v$ instead.\newline
$(2)$ The quotient $\sin(t\slash2)\slash t$ is decreasing with respect to $0<t<2\pi$.\newline
$(3)$ The quotient $t\sin(t)\slash\sin^2(t\slash2)$ is decreasing with respect to $0<t<2\pi$.\newline
$(4)$ The quotient $\log(1+\mu q)\slash\log(1+q)$ is increasing with respect to $q>0$ if $0<\mu\leq1$, and decreasing if $\mu>1$.
\end{proposition}
\begin{proof}
(1) First, define a function $f_1:[0,\pi]\to\R$, $f_1(t)=\sin(t\slash2)-2t$, and note that by calculus $f_1(t)\leq0$ for all $0\leq t\leq\pi$. Now, consider the function $f_2:(0,\pi]\to\R$, $f_2(u)=u\cos(uk)\slash\sin(uk)$ with $0<k<1$. By differentiation and simple trigonometric identities,
\begin{align*}
f_2'(u)
=\frac{\sin(uk\slash2)-2uk}{2\sin^2(uk)}
=\frac{f_1(uk)}{2\sin^2(uk)}\leq0
\end{align*}
so $f_2$ is decreasing with respect to $0<u\leq\pi$. Finally, denote $f_3:(0,1)\to\R$, $f_3(k)=\sin(uk)\slash\sin(vk)$, where $u,v\in(0,\pi]$ are constants. By differentiation,
\begin{align*}
&f_3'(k)=\frac{u\cos(uk)\sin(vk)-v\sin(uk)\cos(vk)}{\sin^2(vk)}\geq0
\quad\Leftrightarrow\quad
\frac{u\cos(uk)}{\sin(uk)}\geq\frac{v\cos(vk)}{\sin(vk)}\\
\Leftrightarrow\quad
&f_2(u)\geq f_2(v)
\quad\Leftrightarrow\quad
u\leq v.
\end{align*}
This is enough to prove the result because $1\slash f_3(1-k)$ is increasing (or decreasing) with respect to $0<k<1$ whenever $f_3$ is.

(2) Define $f_4,f_5:(0,2\pi)\to\R$, $f_4(t)=\sin(t\slash2)$, $f_5(t)=t$. Since $f_4(0)=f_5(0)=0$ and $f_4'(t)\slash f_5'(t)=\cos(t\slash2)\slash2$ is decreasing with respect to $t$, by \cite[Thm B.2, p. 465]{hkvbook}, $f_4(t)\slash f_5(t)$ is decreasing, too.

(3) Denote $f_6:(0,2\pi)\to\R$, $f_6(t)=t\sin(t)\slash\sin^2(t\slash2)$.
By differentiation and some trigonometric identities,
\begin{align*}
f_6'(t)
=\frac{(\sin(t)+t\cos(t))(1-\cos(t))-t\sin^2(t)}{2\sin^4(t\slash2)}
=\frac{(\sin(t)-t)(1-\cos(t))}{2\sin^4(t\slash2)}\leq0,
\end{align*}
so $f_6$ is decreasing for $0<t<2\pi$.

(4) Define $f_7,f_8:(0,\infty)\to\R$, $f_7(q)=\log(1+\mu q)$, $f_8(q)=\log(1+q)$. Note that $f_7(0)=f_8(0)=0$ and $f_7'(q)\slash f_8'(q)=\mu(1+q)\slash(1+\mu q)$ is increasing with respect to $q>0$ if $0<\mu\leq1$ and decreasing if $\mu>1$. By \cite[Thm B.2, p. 465]{hkvbook}, the quotient $f_7(q)\slash f_8(q)$ is increasing (or decreasing) whenever $f_7'(q)\slash f_8'(q)$ is, so the result follows. 
\end{proof}

\begin{theorem}\label{thm_quoQlimits}
For all $0<k<1$ and $0<\theta<2\pi$, the quotient $Q(k,\theta)$ defined in \eqref{def_quoQ} fulfills 
\begin{align*}
&1\leq Q(k,\theta)\leq(\pi\sin(\theta\slash2)\slash\theta)^2,\quad\text{if }\theta<\pi,\\
&(\pi\sin(\theta\slash2)\slash\theta)^2\leq Q(k,\theta)\leq1,\quad\text{if }\theta>\pi.
\end{align*}
\end{theorem}
\begin{proof}
Note that the quotient $Q(k,\theta)$ can be written as $\log(1+q_0(k,\theta))\slash\log(1+q_1(k))$, where
\begin{align}\label{exp_q0q1}
q_0(k,\theta)=\frac{\sin(\theta\slash2)\sin(k\theta\slash2)}{\sin^2((1-k)\theta\slash4)},\quad
q_1(k)=\frac{\sin(k\pi\slash2)}{\sin^2((1-k)\pi\slash4)}.
\end{align}
Clearly,
\begin{align*}
\frac{q_0(k,\theta)}{q_1(k)}=\sin(\theta\slash2)\frac{\sin(k\theta\slash2)}{\sin(k\pi\slash2)}\left(\frac{\sin((1-k)\pi\slash4)}{\sin((1-k)\theta\slash4)}\right)^2 \,.
\end{align*}
Trivially, $q_0(k,\theta)\slash q_1(k)=1$, whenever $\theta=\pi$. It follows from Proposition \ref{prop_triginc}(1) that $q_0(k,\theta)\slash q_1(k)$ is increasing with respect to $k$ if $0<\theta<\pi$, and decreasing if $\pi<\theta<2\pi$. Thus, this quotient is bounded by its limit values
\begin{align*}
\mu_0(\theta)\equiv\lim_{k\to0^+}\frac{q_0(k,\theta)}{q_1(k)}=\frac{\theta\sin(\theta\slash2)}{2\pi\sin^2(\theta\slash4)},\quad
\mu_1(\theta)\equiv\lim_{k\to1^-}\frac{q_0(k,\theta)}{q_1(k)}=\left(\frac{\pi\sin(\theta\slash2)}{\theta}\right)^2.
\end{align*}
By Proposition \ref{prop_triginc}(2)-(3), both of the functions $\mu_0(\theta)$ and $\mu_1(\theta)$ are decreasing with respect to $0<\theta<2\pi$. Since $\mu_0(\pi)=\mu_1(\pi)=1$, this means that the functions $\mu_0$, $\mu_1$ are greater than or equal to 1 for $0<\theta<\pi$, and less than or equal to 1 for $\pi<\theta<2\pi$. It follows that
\begin{equation}\label{equ_forq0}
\begin{aligned}
&q_1(k)\leq\mu_0(\theta)q_1(k)\leq q_0(k,\theta)\leq\mu_1(\theta)q_1(k),
\text{ if }0<\theta<\pi,\text{ and}\\
&\mu_1(\theta)q_1(k)\leq q_0(k,\theta)\leq\mu_0(\theta)q_1(k)\leq q_1(k),
\text{ if }\pi<\theta<2\pi.    
\end{aligned}
\end{equation}
Recall now the expression for the quotient $q_1(k)$ from \eqref{exp_q0q1}. This quotient $q_1(k)$ must be strictly increasing with respect to $0<k<1$, because it has a strictly increasing positive numerator $\sin(k\pi\slash2)$ and a strictly decreasing positive denominator $\sin^2((1-k)\pi\slash4)$. Since $q_1(k)$ has limit values $\lim_{k\to0^+}q_1(k)=0$ and $\lim_{k\to0^+}q_1(k)=\infty$, it maps the interval $(0,1)$ onto $(0,\infty)$. Furthermore, we already earlier noted that $\mu_1(\theta)\geq1$ if $0<\theta<\pi$, and $\mu_1(\theta)\leq1$ if $\pi<\theta<2\pi$. It follows from these observations, inequalities in \eqref{equ_forq0} and Proposition \ref{prop_triginc}(4) that, if $0<\theta<\pi$,
\begin{align*}
\inf_{0<k<1}Q(k,\theta)
&=\inf_{0<k<1}\frac{\log(1+q_0(k,\theta))}{\log(1+q_1(k))}
\geq\inf_{0<k<1}\frac{\log(1+q_1(k))}{\log(1+q_1(k))}=1,\\
\sup_{0<k<1}Q(k,\theta)
&\leq\sup_{0<k<1}\frac{\log(1+\mu_1(\theta)q_1(k))}{\log(1+q_1(k))}
=\lim_{q_1\to0^+}\frac{\log(1+\mu_1(\theta)q_1)}{\log(1+q_1)}
=\mu_1(\theta),
\end{align*}
and, if $\pi<\theta<2\pi$,
\begin{align*}
\inf_{0<k<1}Q(k,\theta)
&\geq\inf_{0<k<1}\frac{\log(1+\mu_1(\theta)q_1(k))}{\log(1+q_1(k))}
=\lim_{q_1\to0^+}\frac{\log(1+\mu_1(\theta)q_1)}{\log(1+q_1)}
=\mu_1(\theta),\\
\sup_{0<k<1}Q(k,\theta)
&\leq\sup_{0<k<1}\frac{\log(1+q_1(k))}{\log(1+q_1(k))}=1,
\end{align*}
which proves the theorem.
\end{proof}

It can be verified that the number 1 in the inequalities of Theorem \ref{thm_quoQlimits} is the best possible constant by showing that it is the limit value of the quotient $Q(k,\theta)$ whenever $k\to1^-$. However, the bound $(\pi\sin(\theta\slash2)\slash\theta)^2$ in Theorem \ref{thm_quoQlimits} does not seem to be sharp. According to several numerical test, the quotient $Q(k,\theta)$ is monotonic with respect to $k$, either decreasing when $0<\theta\leq\pi$ or increasing when $\pi\leq\theta<2\pi$, which would lead into the following result.

\begin{conjecture}\label{conj_quoQlimits}
For all $0<k<1$ and $0<\theta<2\pi$, the quotient $Q(k,\theta)$ fulfills
\begin{align*}
1&=\lim_{k\to1^-}Q(k,\theta)\leq Q(k,\theta)\leq\lim_{k\to0^+}Q(k,\theta)
=\frac{\theta\sin(\theta\slash2)}{2\pi\sin^2(\theta\slash4)},&\text{if }\theta<\pi,\\
\frac{\theta\sin(\theta\slash2)}{2\pi\sin^2(\theta\slash4)}&=\lim_{k\to0^+}Q(k,\theta)\leq Q(k,\theta)\leq\lim_{k\to1^-}Q(k,\theta)
=1,&\text{if }\theta>\pi.
\end{align*}
\end{conjecture}

Finally, we will have the following result.

\begin{corollary}\label{cor_rhodeltainS}
For all $0<\theta<2\pi$ and $x,y\in S_\theta$, the following inequalities hold:
\begin{align*}
(1)\quad&\rho_{S_\theta}(x,y)\leq\delta_{S_\theta}(x,y)\leq(\pi\sin(\theta\slash2)\slash\theta)^2\rho_{S_\theta}(x,y),&\text{if }\theta<\pi,\\
(2)\quad&\delta_{S_\theta}(x,y)=\rho_{S_\theta}(x,y),&\text{if }\theta=\pi,\\
(3)\quad&(\pi\sin(\theta\slash2)\slash\theta)^2\rho_{S_\theta}(x,y)\leq \delta_{S_\theta}(x,y),&\text{if }\theta>\pi.
\end{align*}
\end{corollary}
\begin{proof}
Follows from Corollary \ref{cor_Qwlg}, Theorem \ref{thm_quoQlimits} and the fact that $Q(k,\pi)=1$.
\end{proof}

\begin{remark}\label{rmk_stricterBoundsIfConj}
If Conjecture \ref{conj_quoQlimits} holds, then the coefficient $(\pi\sin(\theta\slash2)\slash\theta)^2$ in Corollary \ref{cor_rhodeltainS} can be replaced by $\theta\sin(\theta\slash2)\slash(2\pi\sin^2(\theta\slash4))$, which gives us even sharper bounds for the M\"obius metric.
\end{remark}

Note that Corollary \ref{cor_rhodeltainS} does not offer an upper bound for the metric $\delta_{S_\theta}$ in terms of $\rho_{S_\theta}$ in the case $\theta>\pi$. As stated in Corollary \ref{cor_Qwlg}, $Q(k,\theta)$ is only a lower limit for the quotient $\delta_{S_\theta}(x,y)\slash\rho_{S_\theta}(x,y)$, so the result of Theorem \ref{thm_quoQlimits} gives us this upper bound. Specifically, even though $Q(k,\theta)\leq1$ for $\pi<\theta<2\pi$, the inequality $\delta_{S_\theta}(x,y)\leq\rho_{S_\theta}(x,y)$ does not hold, as will be shown next.

\begin{lemma}\label{lemma_dOverRhoIfOverPi}
For all $\pi<\theta<2\pi$, there are some points $x,y\in S_\theta$ such that $\delta_{S_\theta}(x,y)>\rho_{S_\theta}(x,y)$.
\end{lemma}
\begin{proof}
For $x=e^{(1-k)\theta i\slash2}$ and $y=e^{(1+k)\theta i\slash2}$ with $0<k<1$, the distance $\rho_{S_\theta}$ is as in the proof of Lemma \ref{lem_Qktheta} and, consequently,
\begin{align*}
\lim_{k\to0^+}\frac{\delta_{S_\theta}(x,y)}{\rho_{S_\theta}(x,y)}
\geq\lim_{k\to0^+}\frac{|0,x,\infty,y|}{\rho_{S_\theta}(x,y)}
=\lim_{k\to0^+}\frac{\log(1+2\sin(k\theta\slash2))}{\log\left(1+\sin(k\pi\slash2)\slash\sin^2((1-k)\pi\slash4)\right)}
=\frac{\theta}{\pi}>1.
\end{align*}
\end{proof}

Finally, we can combine the inequalities of Corollary \ref{cor_rhodeltainS} with our earlier results from Section \ref{sct3} in order to show that Theorem \ref{thm_finalfordeltaS} holds.

\begin{nonsec}{\bf Proof of Theorem \ref{thm_finalfordeltaS}.}
Follows directly from Corollaries \ref{cor_rhodeltaWithS}, \ref{cor_rhodeltaSuni} and \ref{cor_rhodeltainS}. Note that Theorem \ref{thm_finalfordeltaS} only contains the best ones out of the bounds found and, for instance, the lower bound for $\delta_{S_\theta}(x,y)$ in Corollary \ref{cor_rhodeltaWithS}(1) is never better than the one in Corollary \ref{cor_rhodeltainS}(1). 
Similarly, it can be shown that Corollary \ref{cor_rhodeltainS} has always better lower bounds than Corollary \ref{cor_rhodeltaSuni}.
\end{nonsec}

\section{M\"obius metric under quasiregular mappings}\label{sct6}

In this section, we will yet briefly consider the behaviour of the M\"obius metric under $K$-quasiregular mappings. This topic has already been researched in \cite{S99}, see for instance \cite[Thm 5.12, pp. 528-529]{S99}, but we can improve the existing results with our new bounds for the M\"obius metric in sector domains. However, let us first define all the concepts needed.

\begin{definition} \cite[pp. 289-288]{hkvbook}
Choose a domain $G\subset\R^n$ and let the function $f:G\to\R^n$ be ${\rm ACL}^n$, as defined in \cite[p. 150]{hkvbook}. Suppose that there exists a constant $K\geq1$ such that the inequality
\begin{align}\label{f_outerdil}
|f'(x)|^n\leq KJ_f(x),\quad
|f'(x)|=\max_{|h|=1}|f'(x)h|,
\end{align}
where $J_f(x)$ is the Jacobian determinant of $f$ at point $x\in G$, holds a.e. in $G$. Then the function $f$ is called \emph{quasiregular} and the smallest constant $K\geq1$ fulfilling the inequality \eqref{f_outerdil} is the \emph{outer dilatation} of $f$. Similarly, the \emph{inner dilatation} of $f$ is the smallest constant $K\geq1$ such that the inequality
\begin{align}\label{f_innerdil}
J_f(x)\leq K\ell(f'(x))^n,\quad
\ell(f'(x))=\min_{|h|=1}|f'(x)h|
\end{align}
holds a.e. in $G$. The function $f$ is \emph{$K$-quasiregular}, if $\max\{K_I(f),K_O(f)\}\leq K$, where $K_I(f)$ and $K_O(f)$ are the inner and the outer dilatation of $f$, respectively.
\end{definition}

See \cite[(7.1), p. 104]{hkvbook} and \cite{V71} for the definition of the conformal modulus of a curve family $\Gamma$ and denote it by $\M(\Gamma)$. For any non-empty subsets $F_0,F_1\subsetneq\R^n$, let $\Delta(F_0,F_1;\R^n)$ be the family of all the closed non-constant curves joining these two subsets $F_0$ and $F_1$. Furthermore, denote the Euclidean line segment between two points $x,y\in(\R^n\cup\{\infty\})$ by $[x,y]$ and let $e_k$ be the $k$th unit vector of the $n$-dimensional space, $k=1,...,n$. Now, we can define the \emph{Gr\"otzsch capacity} \cite[(7.17), p. 121]{hkvbook} as the decreasing homeomorphism $\gamma_n:(1,\infty)\to (0,\infty)$,
\begin{align*}
\gamma_n(s)=\M(\Delta(\overline{\B}^n,[se_1,\infty];\R^n)),\quad s>1.
\end{align*}
Note that, if $n=2$, we have the following explicit formulas \cite[(7.18), p. 122]{hkvbook}
\begin{align*}
\gamma_2(1/r)=\frac{2\pi}{\mu(r)},\quad \mu(r)=\frac{\pi}{2}\frac{\K(\sqrt{1-r^2})}{\K(r)},\quad
\K(r)=\int^1_0 \frac{dx}{\sqrt{(1-x^2)(1-r^2x^2)}}.
\end{align*}
By using the definition of the Gr\"otzsch capacity, we can define also an increasing homeomorphism $\varphi_{K,n}:[0,1]\to[0,1]$, \cite[(9.13), p. 167]{hkvbook}
\begin{align*}
\varphi_{K,n}(r)=\frac{1}{\gamma_n^{-1}(K\gamma_n(1\slash r))}
\quad\text{for}\quad
0<r<1,\,K>0;\quad
\varphi_{K,n}(0)=0,\quad  
\varphi_{K,n}(1)=1,
\end{align*}
and a number $\lambda_n$ \cite[(9.5) p. 157 \& (9.6), p. 158]{hkvbook} 
\begin{align*}
\log\lambda_n=\lim_{t\to\infty}((\gamma_n(t)\slash\omega_{n-1})^{1\slash(1-n)}-\log t),   
\end{align*}
where $\omega_{n-1}$ is the $(n-1)$-dimensional surface area of the unit sphere $S^{n-1}(0,1)$. For every $n\geq2$, $4\leq\lambda_n<2e^{n-1}$, and $\lambda_2=4$.

The Schwarz lemma is one of the most well-known results in the distortion theory and, while its original version is about the distortion of the Euclidean metric under holomorphic functions, there exists the following modified version of the Schwarz lemma that tells about the distortion of the hyperbolic metric under $K$-quasiregular mappings.

\begin{theorem}\label{thm_schforqr}\cite[Thm 16.2, p. 300 \& Thm 16.39, p. 313]{hkvbook}
Let $G,G'\in\{\uhp^n,\B^n\}$, $f:G\to f(G)\subset G'$ be a non-constant $K$-quasiregular mapping and $\alpha=K_I(f)^{1\slash(1-n)}$. Now,
\begin{align*}
&(1)\quad
{\rm th}\frac{\rho_{G'}(f(x),f(y))}{2}
\leq\varphi_{K,n}\left({\rm th}\frac{\rho_G(x,y)}{2}\right)
\leq\lambda_n^{1-\alpha}\left({\rm th}\frac{\rho_G(x,y)}{2}\right)^\alpha,\\
&(2)\quad
\rho_{G'}(f(x),f(y))
\leq K_I(f)(\rho_G(x,y)+\log4)\\
\intertext{holds for all $x,y\in G$. Furthermore, in the two-dimensional case $n=2$,}
&(3)\quad\rho_{G'}(f(x),f(y))\leq c(K)\max\{\rho_G(x,y),\rho_G(x,y)^{1\slash K}\} 
\end{align*}
for all $x,y\in G$, where
\begin{align*}
c(K)=2{\rm arth}(\varphi_{K,2}({\rm th}(1\slash2)))\leq v(K-1)+K, \quad v=\log(2(1+\sqrt{1-1\slash e^2}))<1.3507,    
\end{align*}
as in \cite[Thm 16.39, p. 313]{hkvbook}. Here, $c(K)\to1$ when $K\to1$ and, by the conformal invariance of the hyperbolic metric, the result $(3)$ also holds for any two simply connected planar domains $G$, $G'$ because they can be mapped conformally onto the unit disk $\B^2$.
\end{theorem}


\begin{corollary}
If $S_\theta$ is a sector with angle $0<\theta\leq\pi$ and $f:S_\theta\to S_\theta$ is a non-constant $K$-quasiregular mapping, then
\begin{align*}
\delta_{S_\theta}(f(x),f(y))
\leq c(K)\min\left\{2,\left(\frac{\pi\sin(\theta\slash2)}{\theta}\right)^2\right\}\max\{\delta_{S_\theta}(x,y),\delta_{S_\theta}(x,y)^{1\slash K}\}   \end{align*}
for all $x,y\in S_\theta$.
\end{corollary}
\begin{proof}
Follows from Theorems \ref{thm_finalfordeltaS}(1)-(2) and \ref{thm_schforqr}(3). 
\end{proof}

A similar result holds for a non-convex sector.

\begin{corollary}\label{cor0_qrNonConvexSector}
If $S_\theta$ is a sector with angle $\pi<\theta<2\pi$ and $f:S_\theta\to S_\theta$ is a non-constant $K$-quasiregular mapping, then
\begin{align*}
\delta_{S_\theta}(f(x),f(y))
\leq4c(K)
\max\left\{\left(\frac{\theta}{\pi\sin(\theta\slash2)}\right)^2\delta_{S_\theta}(x,y),
\left(\frac{\theta}{\pi\sin(\theta\slash2)}\right)^{2\slash K}\delta_{S_\theta}(x,y)^{1\slash K}\right\}   
\end{align*}
for all $x,y\in S_\theta$.
\end{corollary}
\begin{proof}
Follows from Theorems \ref{thm_finalfordeltaS}(3) and \ref{thm_schforqr}(3).
\end{proof}

\begin{corollary}\label{cor1_qrNonConvexSector}
If $S_\theta$ is a sector with angle $\pi<\theta<2\pi$ and $f:S_\theta\to S_\theta$ is a non-constant $K$-quasiregular mapping, then
\begin{align*}
{\rm th}\frac{\delta_{S_\theta}(f(x),f(y))}{4}
\leq c(K)\frac{\theta}{\pi}\left(2{\rm th}\frac{\delta_{S_\theta}(x,y)}{2}\right)^{1\slash K}
\end{align*}
for all $x,y\in S_\theta$.
\end{corollary}
\begin{proof}
By Theorems \ref{thm_finalfordeltaS}(3), \ref{thm_schforqr}(3) and \ref{thm_rhos}(3),
\begin{align*}
{\rm th}\frac{\delta_{S_\theta}(f(x),f(y))}{4}
&\leq\frac{\theta}{\pi}{\rm th}\frac{\rho_{S_\theta}(f(x),f(y))}{2}
\leq\frac{\theta}{\pi}{\rm th}\left(\frac{c(K)}{2}\max\{\rho_{S_\theta}(x,y),\rho_{S_\theta}(x,y)^{1\slash K}\}\right)\\
&\leq\frac{\theta}{\pi}{\rm th}\left(\frac{c(K)}{2}\max\{2{\rm arth}(s_{S_\theta}(x,y)),(2{\rm arth}(s_{S_\theta}(x,y)))^{1\slash K}\}\right).
\end{align*}
It follows from \cite[Thm 5.3, p. 11]{sqm} that
\begin{align*}
{\rm th}\left(\frac{C}{2}\max\{2{\rm arth}(t),(2{\rm arth}(t))^{1\slash K}\}\right)\leq Ct^{1\slash K}  
\end{align*}
for all $0<t<1$, $K\geq1$ and $C\geq1$. Consequently,
\begin{align*}
{\rm th}\frac{\delta_{S_\theta}(f(x),f(y))}{4}
\leq c(K)\frac{\theta}{\pi}s_{S_\theta}(x,y)^{1\slash K} \,.
\end{align*}
By Corollary \ref{cor_deltasj}(2), $s_{S_\theta}(x,y)\leq2{\rm th}(\delta_{S_\theta}(x,y)\slash2)$, so the result follows.
\end{proof}

\begin{remark}
Neither Corollary \ref{cor0_qrNonConvexSector} nor Corollary \ref{cor1_qrNonConvexSector} offers a bound for the distortion that is always better than the result of the other corollary, which can be seen by studying the case where $\theta\to\pi^-$ and $c(K)=K=1$ for varying points $x,y\in S_\theta$.
\end{remark}

\begin{corollary}\label{cor_qrBtoSector}
If $S_\theta$ is a sector with angle $0<\theta<2\pi$ and $f:\B^2\to S_\theta$ is a non-constant $K$-quasiregular mapping, then, for all $x,y\in\B^2$,
\begin{align*}
&\delta_{S_\theta}(f(x),f(y))
\leq c(K)\min\left\{2,\left(\frac{\pi\sin(\theta\slash2)}{\theta}\right)^2\right\}\max\{\delta_{\B^2}(x,y),\delta_{\B^2}(x,y)^{1\slash K}\},\\
&\text{if}\quad0<\theta\leq\pi,\quad\text{and}\\
&\delta_{S_\theta}(f(x),f(y))
\leq4c(K)\max\{\delta_{\B^2}(x,y),\delta_{\B^2}(x,y)^{1\slash K}\}
\quad\text{if}\quad\pi<\theta<2\pi.
\end{align*}
\end{corollary}
\begin{proof}
Follows from Theorems \ref{thm_seitminvariant}, \ref{thm_finalfordeltaS} and \ref{thm_schforqr}(3). 
\end{proof}

\begin{corollary}\label{cor_6.13}
If $S_\theta$ is a sector with angle $0<\theta\leq\pi$ and $f:\B^2\to S_\theta$ is a non-constant $K$-quasiregular mapping, then for all $x\in\B^2$ such that $|x|\geq(e-1)\slash(e+1)$,
\begin{align*}
|f(x)|\leq|f(0)|\left(\frac{1+|x|}{1-|x|}\right)^{c(K)u(\theta)}
\quad\text{with}\quad
u(\theta)=\min\left\{2,\left(\frac{\pi\sin(\theta\slash2)}{\theta}\right)^2\right\}.   
\end{align*}
\end{corollary}
\begin{proof}
By the triangle inequality, Theorems \ref{thm_deltaj}, \ref{thm_finalfordeltaS} and \ref{thm_schforqr}(3), and \cite[(4.14), p. 55]{hkvbook}, 
\begin{align*}
\log\frac{|f(x)|}{|f(0)|}
&\leq\log\frac{|f(x)-f(0)|+|f(0)|}{|f(0)|}
=\log\left(1+\frac{|f(x)-f(0)|}{|f(0)|}\right)\\
&\leq\log\left(1+\frac{|f(x)-f(0)|}{\min\{d_{S_\theta}(f(x)),d_{S_\theta}(f(0))\}}\right)
=j_{S_\theta}(f(x),f(0))
\leq\delta_{S_\theta}(f(x),f(0))\\
&\leq u(\theta)\rho_{S_\theta}(f(x),f(0))
\leq c(K)u(\theta)\max\{\rho_{\B^2}(x,0),\rho_{\B^2}(x,0)^{1\slash K}\}\\
&=c(K)u(\theta)\max\left\{\log\frac{1+|x|}{1-|x|},\left(\log\frac{1+|x|}{1-|x|}\right)^{1\slash K}\right\},
\end{align*}
and, if the inequality
\begin{align*}
\log\frac{1+|x|}{1-|x|}\geq\left(\log\frac{1+|x|}{1-|x|}\right)^{1\slash K}
\quad\Leftrightarrow\quad
\frac{1+|x|}{1-|x|}\geq e
\quad\Leftrightarrow\quad
|x|\geq\frac{e-1}{1+e}
\end{align*}
holds, then we will have
\begin{align*}
\log\frac{|f(x)|}{|f(0)|}
\leq c(K)u(\theta)\log\frac{1+|x|}{1-|x|}
\quad\Leftrightarrow\quad
|f(x)|\leq|f(0)|\left(\frac{1+|x|}{1-|x|}\right)^{c(K)u(\theta)}.
\end{align*}
\end{proof}

\begin{remark}
Note that Corollary \ref{cor_6.13} refines \cite[Thm 16.19(1), p. 306]{hkvbook} when $n=2$.
\end{remark}
\section{M\"obius metric in polygon}\label{sct7}

In this section, we will introduce a few open questions related to the M\"obius metric inside a polygon domain. Especially, we are interested in the inequality between the M\"obius metric and the hyperbolic metric defined in a polygon. All the computational findings and Figure \ref{fig3} have been made with MATLAB programs from \cite{nrv1}.

Even though the inequality $\rho_{S_\theta}(x,y)\leq\delta_{S_\theta}(x,y)$ holds in all convex sectors by Theorem \ref{thm_finalfordeltaS} and these metrics are equivalent in such convex domains as the unit disk and the upper half-plane, our computer experiments verify that neither of metrics is always greater than or equal to the other in all polygonal domains, not even in all convex polygons.

\begin{conjecture}\label{conj_convexIne}
If $G\subsetneq\R^2$ is any bounded polygonal domain, there are always some points $x,y,u,v\in G$ such that $\rho_G(x,y)<\delta_G(x,y)$ and $\rho_G(u,v)>\delta_G(u,v)$.
\end{conjecture}

However, the values of these two metrics do not differ very much from each other in the domain $G_k$ shaped like a regular convex $k$-gon with $k$ vertices $e^{p2\pi i/k}$, $p=0,1,...,k-1$, especially as the value of $k$ grows and this domain resembles more and more the unit disk, where these metrics are equivalent.

\begin{conjecture}
If $G_k$ is the above regular $k$-gon and $x,y\in\cap G_k$ are distinct points, then
\begin{align*}
\lim_{k\to\infty}(\delta_{G_k}(x,y)\slash\rho_{G_k}(x,y))\to1.    
\end{align*}
\end{conjecture}

Furthermore, recall from Corollary \ref{cor_deltauni} that the inequality $\delta_G(x,y)\leq4\rho_G(x,y)$ holds for all points $x,y$ in a simply-connected domain $G$, and our computer tests suggests that this result can be improved in the case of polygonal domains.

\begin{figure}
    \centering
    \includegraphics[scale=0.8]{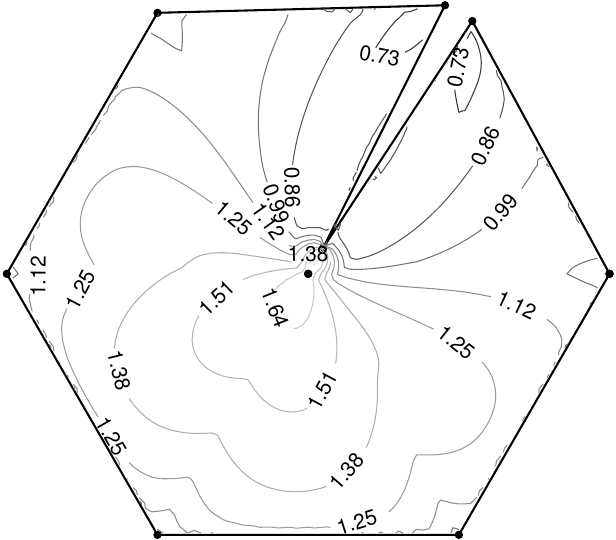}
    \caption{Contour plot of the quotient $\delta_G(x,y)\slash\rho_G(x,y)$, when $G$ is the polygon with vertices $1,e^{0.95\pi i\slash3},0.1e^{\pi i\slash3},e^{1.05\pi i\slash3},e^{2\pi i\slash3},e^{\pi i},e^{4\pi i\slash3},e^{5\pi i\slash3}$, the point $x$ is fixed as 0, and $y$ varies inside $G$.}
    \label{fig3}
\end{figure}

\begin{conjecture}\label{conj_cons}
For all polygonal domains $G\subsetneq\R^2$, the inequality $\rho_G(x,y)\slash 2\leq\delta_G(x,y)\leq2\rho_G(x,y)$ holds for all $x,y\in G$.
\end{conjecture}

Figure \ref{fig3} contains an example of a non-convex polygon where the values of the quotient $\delta_G(x,y)\slash\rho_G(x,y)$ vary at least on the interval $[0.73,1.64]$. Note that, based on our computer tests, the latter constant in the inequality of Conjecture \ref{conj_cons} can be replaced with a smaller one when considering only convex domains. Another interesting notion is that by Corollary \ref{cor_deltauni} the uniformity constant $A_G$ of a domain $G$ is fulfills
\begin{align*}
A_G\geq\rho_G(x,y)\slash(2\delta_G(x,y)   
\end{align*}
for all points $x,y$ in the domain $x,y$, so computing the maximum value of the quotient $\rho_G(x,y)\slash(2\delta_G(x,y)$ gives an lower bound for the uniformity constant $A_G$.

\def\cprime{$'$} \def\cprime{$'$} \def\cprime{$'$}
\providecommand{\bysame}{\leavevmode\hbox to3em{\hrulefill}\thinspace}
\providecommand{\MR}{\relax\ifhmode\unskip\space\fi MR }
\providecommand{\MRhref}[2]{%
  \href{http://www.ams.org/mathscinet-getitem?mr=#1}{#2}
}
\providecommand{\href}[2]{#2}

\end{document}